\documentclass[11pt,reqno]{article}

\usepackage[utf8]{inputenc}

\usepackage[a4paper, margin=2.5cm,top=3cm, bottom=3cm]{geometry}
\usepackage{amsmath, amsthm, amssymb, amsfonts}
\usepackage{mathtools}
\usepackage{hyperref}
\usepackage{enumitem}
\usepackage{setspace}
\usepackage{xcolor}
\usepackage{thmtools}
\usepackage{thm-restate}
\usepackage{verbatim}

\theoremstyle{plain}
\newtheorem{theorem}{Theorem}[section]
\newtheorem{lemma}[theorem]{Lemma}

\theoremstyle{definition}

\newcommand*\samethanks[1][\value{footnote}]{\footnotemark[#1]}

\title{Note on Long Paths in Eulerian Digraphs}

\author{
	Charlotte Knierim\thanks{Department of Computer Science, ETH Zurich, Switzerland\newline $\{$cknierim$\vert$larcherm$\vert$anders.martinsson$\}$@inf.ethz.ch\newline
	}
	\and
	Maxime Larcher\samethanks[1]
	\and Anders Martinsson\samethanks[1]}

\hypersetup{
	colorlinks,
	linkcolor={red!60!black},
	citecolor={green!50!black},
	urlcolor={blue!80!black}
}

\newcommand{\din}{\ensuremath{\mathrm{d}^-}}
\newcommand{\dout}{\ensuremath{\mathrm{d}^+}}

\newcommand{\wout}{\ensuremath{\mathrm{w}^+}}

\newcommand{\weight}{\ensuremath{\mathrm{w}}}

\begin{document}
\maketitle
\begin{abstract}

Long paths and cycles in Eulerian digraphs have received a lot of attention recently. In this short note, we show how to use methods from \cite{KLMN21} to find paths of length $d/(\log d+1)$ in Eulerian digraphs with average degree $d$, improving  the recent result of $\Omega(d^{1/2+1/40})$. Our result is optimal up to at most a logarithmic factor.

\end{abstract}
\section{Introduction}

One of the fundamental questions of extremal combinatorics is to determine how `large' a graph $G$ may be before it needs to contain certain graphs as subgraphs. A famous result in this topic is a theorem of Erd\H{o}s and Gallai~\cite{EG59}, which states that any graph of average degree $d$ contains a path (resp.\ a cycle) of length linear in $d$.

This same question turns out to be more difficult for digraphs, even in the apparent simple case of a path. In all generality, one cannot hope for a statement similar to that of Erd\H{o}s and Gallai, as there are digraphs --- for instance the complete bipartite graph $K_{\frac{n}{2}, \frac{n}{2}}$ in which all edges are oriented in the same direction ---  which have high average degree but only paths of length $1$. Part of the problem hence resides in determining for which class of digraphs one may hope to prove a counterpart to Erd\H{o}s and Gallai's theorem. For this, Bollob\'as and Scott~\cite{BS96} conjectured that all Eulerian digraphs with average degree\footnote{Throughout this note, average degree stands for \emph{average out-degree}. In particular, for an $n$-vertex, $m$-edge digraph we have $d = m/n$.} $d$ contain a path of length $\Omega(d)$. A first partial answer was given by Huang, Ma, Shapira, Sudakov and Yuster~\cite{HMSSY13} who proved a lower bound of $d^{1/2}$. Very recently Janzer, Sudakov and Tomon~\cite{JST21} improved this bound to $\Omega(d^{1/2+1/40})$. We push this to $d/ (\log d+1)$. As any regular tournament on $2k+1$ vertices is Eulerian with average degree $k$ and longest path of length $2k$ (in fact, as is easily shown by induction, \emph{any} tournament has a Hamilton path), our result is optimal up to at most a logarithmic factor.
\begin{theorem}
\label{thm:longpath}
Let $G$ be an Eulerian digraph with average degree $d$. Then $G$ has a path of length at least ${d}/(\log d + 1)$.
\end{theorem}

To attain this bound, we look at a related problem, the problem of determining the minimum number of paths into which $G$ may be decomposed. The connection is the following: if one can decompose $G$ into few paths, then one of them needs be long; if one can find long paths in graphs, then one may sequentially take out long paths from $G$ to find a small decomposition. In~\cite{KLMN21}, the authors together with A.\ Noever worked towards a conjecture of Haj\'os and proved that any Eulerian digraph may be decomposed into $O( n \log \Delta)$ cycles, where $\Delta$ denotes the maximum degree of the graph. As noted by Janzer, Sudakov and Tomon, this implies a lower bound of $ \Omega( d / \log \Delta )$ on the length of a path, which is below their bound when $\Delta$ is (very) large compared to $d$. 

In this note, we adapt our ideas from~\cite{KLMN21} to paths. We give an upper bound on the number of paths required to decompose an Eulerian digraph. From this, we improve the bound of Janzer, Sudakov and Tomon.

\begin{theorem}
Let $G$ be an $n$-vertex Eulerian digraph with average degree $d$. Then we can decompose $G$ into at most $n(\log d + 1)$ directed paths.
\label{thrm:main}
\end{theorem}

Theorem~\ref{thm:longpath} is a direct consequence of Theorem~\ref{thrm:main}.
The core ideas of our proofs are based on a result by Bollob\'as and Scott~\cite{BS96} (Corollary 2) and the `Cycle Removal Lemma' from~\cite{KLMN21} (Lemma 2.2). 
\section{Proof}
The following lemma is a combination of Theorem 1 and Corollary 2 from \cite{BS96}. Note that the graph we look at is almost strongly connected - we know that all components of an Eulerian digraph are strongly connected and the graph we consider is at most one edge away from being Eulerian. 
\begin{lemma}
\label{lemma:find_path}
Let $G=(V,E)$ be a digraph and let $s,r\in V$ be vertices in $G$. Let $\weight:E\rightarrow \mathbb{R}_{\geq 0}$ be an edge-weighting of $G$. For a vertex $v\in V$, we let $\wout(v)$ denote the total weight of outgoing edges from $v$. If the following statements hold:
\begin{enumerate}
    
    \item Either $G$ is Eulerian or we have $\dout(r)=\din(r)-1$, $ \dout(s)= \din(s)+1$, and $\dout(v) = \din (v)$ for all $v \in V\setminus\{s,r\}$, \label{lem:itm:cond1}
    \item for all $ v\in V\setminus\{r\}$ we have $\wout (v)\ge 1$,
\end{enumerate}
then $G$ has a path of weight at least 1 that ends in $r$.
\end{lemma}

\begin{proof}
For the readability of this proof, we assume that $G$ is weakly connected. If this was not the case, all the conditions hold for every component and we can restrict ourselves to the component containing $r$.  Our proof is based on the depth-first search algorithm (DFS); we recall that this algorithm starts at a fixed root vertex and, at each step, moves to an (arbitrary) unvisited neighbour of the current vertex. If no such neighbour is available it backtracks to the last visited vertex which has an unvisited neighbour and stops once all vertices have been visited. See e.g.~\cite{cormen2009algorithms} for an introduction of this algorithm.

We start a ``backwards DFS'' from $r$ as follows. We use edges in the opposite direction (from target to source) and whenever we have multiple edges to choose from we take the heaviest edge first, breaking ties arbitrarily. Let $u$ be the last vertex visited by the DFS in the component of $r$. Let $P$ be the path $u - r$ in the DFS tree. Then $P$ has weight at least 1.
This can be easily seen as follows.

We first want to argue that the DFS reached all vertices. For this, we need to argue that from every vertex we can find a path to $r$. Note that if $G$ is not Eulerian, it can easily be made Eulerian by adding an edge from $r$ to $s$. This (multi-)graph is strongly connected. But as the edge from $r$ to $s$ is not contained in any path to $r$ we know that in $G$ every vertex has a path to $r$ and thus our DFS visits all vertices.

\begin{figure}[t]
\centering
\includegraphics[scale=0.8]{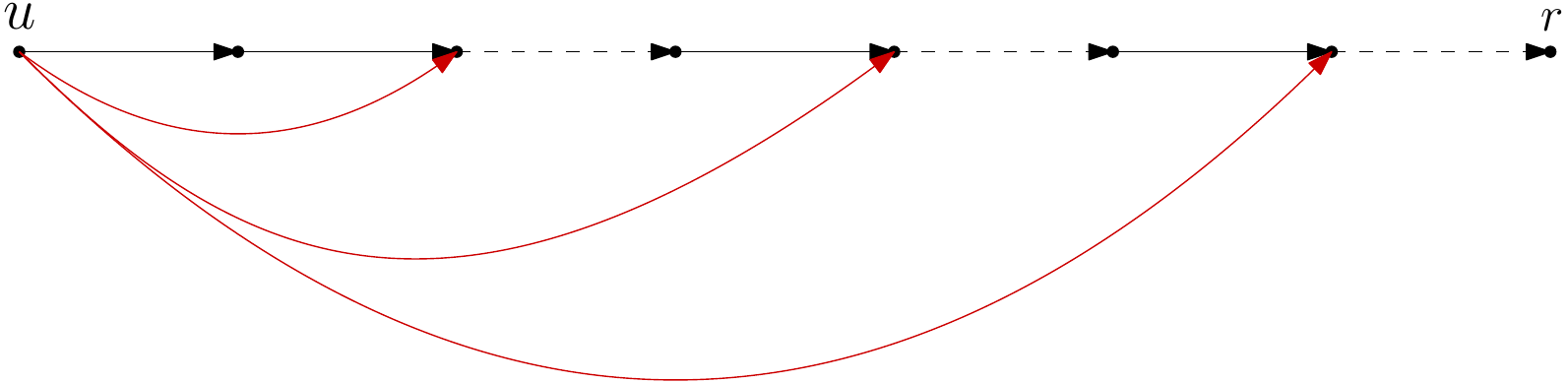}
\caption{Illustration of the $u$-$r$ path from the proof of Lemma~\ref{lemma:find_path}.}
\label{fig:1}
\end{figure}
Now we want to argue that the path $P$ has indeed weight at least one. First, note that all out-neighbours of $u$ lie on $P$. If this was not the case, then $u$ would have been visited earlier (having this out-neighbour as a predecessor). 
Then every vertex that is an out-neighbour of $u$ has the property that its in-edge used on the path is at least as heavy as its edge coming from $u$. This is because our DFS always picks the heaviest edge first. Then just adding up these edges (black edges in Figure~\ref{fig:1}) gives a weight of at least one. As edge weights are non-negative we conclude that the weight of $P$ is at least 1.

\end{proof}

The proof of the main result goes along the lines of the proof of Lemma~2.2 from \cite{KLMN21}.
\begin{proof}[Proof of Theorem~\ref{thrm:main}]

We remove the paths from $G$ in an iterative fashion to form a decreasing sequence $G=G_0\supset G_1\supset \dots$. At each step $t$, $G_t$ is either Eulerian---in particular, this is the case for $G_0 = G$---in which case we choose $s_t, r_t$ arbitrarily, or there are unique vertices $s_t, r_t$ such that $\dout_{G_t}(r_t)=\din_{G_t}(r_t)-1$, $\dout_{G_t}(s_t)=\din_{G_t}(s_t)+1$ and $\dout_{G_t}(v)=\din_{G_t}(v)$ for all $v\in V\setminus\{s_t, r_t\}$. We apply Lemma~\ref{lemma:find_path} to $G_t$ with $s = s_t, r = r_t$ and with \emph{uniform edge weighting} $\weight_{G_t}(vw) = 1 / \dout_{G_t}(v)$ for each edge $(vw) \in G_t$. Given the resulting path $P_t$, we put $G_{t+1}=G_t\setminus P_t$; since $P_t$ ends in $r_t$, we observe that Condition~\ref{lem:itm:cond1} of Lemma~\ref{lemma:find_path} is preserved for the next iteration.

Let $T$ denote the number of paths removed before $G_T$ is empty. We show by a double-counting argument on the sum $\sum_{t=0}^{T-1}{\weight_{G_t}(P_t)}$ that $T \le n ( \log d + 1 )$. 
Clearly as $P_t$ was chosen using Lemma~\ref{lemma:find_path}, it has weight $\weight_{G_t}(P_t) \ge 1$, so that sum is at least $T$. For the upper bound, we observe that the contribution of each vertex $v$ to the sum is $1 / \dout_G(v)$ the first time a path uses an out-edge of $v$, $1 / (\dout_G(v) - 1)$ the second time, and so on. Hence we have 
\begin{align*}
    T \le \sum_{t = 0}^{T-1}{ \weight_{G_t}(P_t) } = \sum_{v \in V}{ \sum_{i = 1}^{\dout_G(v)}{ 1 / i} }.
\end{align*}
Using the fact that $\sum_{i = 1}^{k}{ 1/i } \le \log k + 1$ for all $k$, and Jensen's inequality, we conclude that $G$ may be decomposed using at most $T \le n(\log d + 1)$ directed paths.

\end{proof}

\bibliographystyle{abbrv}
\bibliography{sources}
\end{document}